\documentclass[12pt]{amsart}

\usepackage{amsthm}
\usepackage{amsmath}
\usepackage{amssymb}
\usepackage{amscd}
\usepackage{latexsym}
\usepackage{mathrsfs}
\usepackage[colorlinks=true]{hyperref}

\usepackage[dvips]{color,graphicx}

\usepackage{enumerate}

\usepackage[margin=1.56in]{geometry}

\theoremstyle{plain}
  \newtheorem{theorem}{Theorem}[section]
  \newtheorem{proposition}[theorem]{Proposition}
  \newtheorem{lemma}[theorem]{Lemma}
  \newtheorem{corollary}[theorem]{Corollary}
  
\theoremstyle{definition}
  
  \newtheorem{example}[theorem]{Example}
  
  \newtheorem{question}[theorem]{Question}
 \theoremstyle{remark}
  \newtheorem{remark}[theorem]{Remark}

\numberwithin{equation}{section}

\newcommand{\reals}{{\mathbb R}}

\newcommand{\NN}{{\mathbb N}}
\newcommand{\flag}{{\mathcal F}}
\newcommand{\DD}{\Delta}
\newcommand{\A}{\mathbf{a}}
\newcommand{\B}{\mathbf{b}}
\newcommand{\C}{\mathbf{c}}
\newcommand{\D}{\mathbf{d}}

\newcommand{\conv}{\mathrm{conv}}
\newcommand{\pyr}{\mathrm{Pyr}}
\newcommand{\prism}{\mathrm{Prism}}
\newcommand{\bipyr}{\mathrm{Bipyr}}

\begin{document}

\title[]
      {Flag enumerations of matroid base polytopes}
\author{Sangwook Kim}
\address{Department of Mathematical Sciences\\
         George Mason University\\
         Fairfax, VA 22030, USA}
\email{skim22@gmu.edu}

\begin{abstract}   
   In this paper, we study flag structures of matroid base polytopes.
   We describe faces of matroid base polytopes in terms of matroid data, 
   and give conditions for hyperplane splits of matroid base polytopes.
   Also, we show how the $\C\D$-index of a polytope can be expressed 
   when a polytope is split by a hyperplane, and 
   apply these to the $\C\D$-index of a matroid base polytope of a rank $2$ matroid.
\end{abstract}

\maketitle

\section{Introduction}
\label{sec-introduction}

   For a matroid $M$ on $[n]$, a \emph{matroid base polytope} $Q(M)$ is
   the polytope in $\reals^n$ whose vertices are the incidence vectors of the bases of $M$.
   The polytope $Q(M)$ is a face of a \emph{matroid polytope} first studied by 
   Edmonds~\cite{Edmonds}, whose vertices are the incidence vectors of \emph{all}
   independent sets in $M$.
   In this paper, we study flags of faces and the $\C\D$-index
   of matroid base polytopes.
   
   It is known that a face $\sigma$ of a matroid base polytope is 
   the matroid base polytope $Q(M_\sigma)$ for some matroid $M_\sigma$ on $[n]$
   (see \cite{FeichtnerSturmfels} and Section~\ref{sec-matroid-base-polytopes} below).
   We show that $M_\sigma$ can be described using 
   equivalence classes of factor-connected flags of subsets of $[n]$.
   As a result, one can describe faces of $Q(M)$ in terms of matroid data:
   \begin{theorem}
   [Theorem~\ref{theorem-posets-for-faces-of-matroid-polytopes}]
      Let $M$ be a connected matroid on a ground set $[n]$.
      For a face $\sigma$ of the matroid base polytope $Q(M)$, one can associate a poset $P_\sigma$ 
      defined as follows:
      \begin{enumerate}[(i)]
         \item
         the elements of $P_\sigma$ are the connected components of $M_\sigma$, and
         \item
         for distinct connected components $C_1$ and $C_2$ of $M_\sigma$,
         $C_1 < C_2$ if and only if 
         $$
         C_2 \subseteq S \subseteq [n] \text{ and }
         \sigma \subseteq H_S \text{ implies } C_1 \subseteq S,
         $$
         where $H_S$ is the hyperplane in $\reals^n$ defined by $\sum_{e \in S} x_e = r(S)$.
      \end{enumerate}
   \end{theorem}
   
   The $\C\D$-index $\Psi(Q)$ of a polytope $Q$, a polynomial in the noncommutative variables
   $\C$ and $\D$, is a very compact encoding of the flag numbers of a polytope $Q$~\cite{BayerKlapper}.
   Ehrenborg and Readdy~\cite{EhrenborgReaddy} express the $\C\D$-indices of a prism, a pyramid, and
   a bipyramid of a polytope $Q$ in terms of $\C\D$-indices of $Q$ and its faces.
   Also, the $\C\D$-index of zonotopes, a special class of polytopes, is well-understood
   \cite{BilleraEhrenborgReaddy97, BilleraEhrenborgReaddy98}.
   Generalizing the formulas of the $\C\D$-indices of a prism and a pyramid of a polytope, 
   we show how the $\C\D$-index of a polytope can be expressed when a polytope is split by a hyperplane
   in Section~\ref{sec-cd-index}.
   %
   
   In Section~\ref{sec-hyperplane-splits}, 
   we find the conditions when a matroid base polytope is split into two 
   matroid base polytopes by a hyperplane:
   \begin{theorem}
   [Theorem~\ref{matroid-base-polytope-split}]
      Let $M$ be a rank $r$ matroid on $[n]$ and $H$ be a hyperplane in $\reals^n$ 
      given by $\sum_{e \in S} x_e = k$.
      Then $H$ decomposes $Q(M)$ into two matroid base polytopes if and only if
      \begin{enumerate}[(i)]
         \item
         $r(S) > k$ and $r(S^c) > r-k$,
         \item
         if $I_1$ and $I_2$ are $k$-element independent subsets of $S$ such that
         $(M/I_1)|_{S^c}$ and $(M/I_2)|_{S^c}$ have rank $r-k$, then 
         $(M/I_1)|_{S^c} = (M/I_2)|_{S^c}.$
      \end{enumerate}
   \end{theorem}
   
   We apply this theorem to the $\C\D$-index of matroid base polytopes for rank $2$ matroids
   in Section~\ref{sec-rank-2-matroids}.
   
   Throughout this paper, we assume familiarity with the basic concepts of matroid theory.
   For further information, we refer the readers to \cite{Oxley}.
   
\section{Matroid base polytopes}
\label{sec-matroid-base-polytopes}

   This section contains the description of faces of matroid base polytopes.
   In particular, we associate a poset to each face of a matroid base polytope.
   
   We start with a precise characterization of matroid base polytopes.
   Let $\mathcal{B}$ be a collection of $r$-element subsets of $[n]$.
   For each subset $B = \{ b_1, \dots, b_r \}$, let
   $$
   e_B = e_{b_1} + \cdots + e_{b_r} \in \reals^n ,
   $$
   where $e_i$ is the $i$th standard basis vector of $\reals^n$.
   The collection $\mathcal{B}$ is represented by the convex hull of these points
   $$
   Q(\mathcal{B}) = \conv \{ e_B : B \in \mathcal{B} \}.
   $$
   This is a convex polytope of dimension $\le n-1$ 
   and is a subset of the $(n-1)$-simplex
   $$
   \DD_n = \{ (x_1, \dots, x_n) \in \reals^n : x_1 \ge 0, 
         \dots, x_n \ge 0, x_1 + \cdots + x_n = r \}.
   $$
   Gelfand, Goresky, MacPherson, and Serganova~\cite[Thm. 4.1]{GelfandGoreskyMacPhersonSerganova} 
   show the following characterization of matroid base polytopes.
   \begin{theorem}
   \label{thm-matroid-polytopes}
      $\mathcal{B}$ is the collection of bases of a matroid if and only if
      every edge of the polytope $Q(\mathcal{B})$ is parallel to
      a difference $e_\alpha - e_\beta$ of two distinct standard basis vectors.
   \end{theorem}
   
   For a rank $r$ matroid $M$ on a ground set $[n]$ 
   with a set of bases $\mathcal{B}(M)$,
   the polytope $Q(M) := Q(\mathcal{B}(M))$ is called 
   the \emph{matroid base polytope} of $M$.
   
   By the definition, the vertices of $Q(M)$ represent the bases of $M$.
   For two bases $B$ and $B'$ in $\mathcal{B}(M)$, 
   $e_B$ and $e_{B'}$ are connected by an edge if and only if 
   $e_B - e_{B'} = e_\alpha - e_\beta$ for some $\alpha, \beta \in [n]$.
   Since the latter condition is equivalent to $B \setminus B' = \{ \alpha \}$
   and $B' \setminus B = \{ \beta \}$, the edges of $Q(M)$ represent 
   the basis exchange axiom.
   The basis exchange axiom gives the following equivalence relation 
   on the ground set $[n]$ of the matroid $M$:
   $\alpha$ and $\beta$ are \emph{equivalent} if there exist bases $B$ and $B'$ in $\mathcal{B}(M)$
   with $B \setminus B' = \{ \alpha \}$ and $B' \setminus B = \{ \beta \}$.
   The equivalence classes are called the \emph{connected components} of $M$.
   The matroid $M$ is called \emph{connected} if it has only one connected
   component.
   Feichtner and Sturmfels~\cite[Prop. 2.4]{FeichtnerSturmfels} express the dimension of 
   the matroid base polytope $Q(M)$ in terms of the number of connected
   components of $M$.
   
   \begin{proposition}
   \label{dim-of-matroid-polytope}   
      Let $M$ be a matroid on $[n]$.
      The dimension of the matroid base polytope $Q(M)$ equals $n - c(M)$, where
      $c(M)$ is the number of connected components of $M$.
   \end{proposition}
   
   Theorem~\ref{thm-matroid-polytopes} implies that every face of a matroid base polytope
   is also a matroid base polytope.
   For a face $\sigma$ of $Q(M)$, let $M_\sigma$ denote the matroid on $[n]$ 
   whose matroid base polytope is $\sigma$.
   For $\omega \in \reals^n$, let $M_\omega$ denote the matroid whose bases
   $\mathcal{B}(M_\omega)$ is the collection of bases of $M$ having minimum $\omega$-weight.
   Then $Q(M_\omega)$ is the face of $Q(M)$ at which the linear form 
   $\sum_{i=1}^n \omega_i x_i$ attains its minimum.
   Let $\flag(\omega)$ denote the unique flag of subsets
   $$
   \{ \emptyset =: S_0 \subseteq S_1 \subseteq \cdots \subseteq S_k \subseteq S_{k+1} := [n] \}
   $$
   for which $\omega$ is constant on each set $S_i \setminus S_{i-1}$
   and $\omega|_{S_i \setminus S_{i-1}} < \omega|_{S_{i+1} \setminus S_i}$.
   Ardila and Klivans~\cite{ArdilaKlivans}
   show that $M_\omega$ depends only on $\flag(\omega)$, and hence
   one can call it $M_{\flag}$.
   They also give the following description of $M_\flag$.
   
   \begin{proposition}
   \emph{\cite[Prop. 2]{ArdilaKlivans}}
   \label{bases-of-M-F}
      Let $M$ be a matroid on $[n]$ and $\flag$ be a flag of subsets
      $$
      \{ \emptyset =: S_0 \subseteq S_1 \subseteq \cdots \subseteq S_k \subseteq S_{k+1} := [n] \},
      $$
      then 
      $$
      M_{\flag} = \bigoplus_{i=1}^{k+1} (M |_{S_i}) / S_{i-1},
      $$
      where $M|_S$ is the restriction of $M$ on $S$ and $M/S$ is the contraction of $M$ on $S$.
   \end{proposition}
   
   A flag 
   $\flag = \{ \emptyset =: S_0 \subseteq S_1 \subseteq 
   \cdots \subseteq S_k \subseteq S_{k+1} := [n] \}$ 
   is called \emph{factor-connected} (with respect to $M$) if the matroids
   $(M|_{S_{i}})/ S_{i-1}$ are connected for all $i = 1, \dots, k+1$.
   If $\flag$ is factor-connected, then the connected components of $M_\flag$ are
   $S_1 \setminus S_0, S_2 \setminus S_1, \dots, S_{k+1} \setminus S_k$.
   Proposition~\ref{dim-of-matroid-polytope} and Proposition~\ref{bases-of-M-F}
   show that the dimension of $Q(M_{\flag})$ is $n-k-1$ 
   when $\flag$ is factor-connected. 
   
   Since $Q(M_1 \oplus M_2) = Q(M_1) \times Q(M_2)$, it is enough to
   restrict to attention to connected matroid $M$.
   For a connected matroid $M$ on $[n]$, facets of $Q(M)$ correspond to factor-connected flags of the form
   $\{ \emptyset \subseteq S \subseteq [n] \}$.
   Feichtner and Sturmfels~\cite{FeichtnerSturmfels} show that there are two types of facets of $Q(M)$:
   \begin{enumerate}[(i)]
      \item
      a facet corresponding to a factor-connected flag $\{ \emptyset \subseteq F \subseteq [n] \}$ 
      for some flat $F$ of $M$ (in this case, the facet is called a \emph{flacet}),
      \item
      a facet corresponding to a factor-connected flag $\{ \emptyset \subseteq S \subseteq [n] \}$ 
      for an $(n-1)$-subset $S$ of $[n]$.
   \end{enumerate}
   
   %
   %
   \begin{lemma}
   \label{switching-components}
      Let $M$ be a connected matroid on $[n]$ and
      $$
      \flag = \{ \emptyset =: S_0 \subseteq S_1 \subseteq \cdots \subset
      S_k \subseteq S_{k+1} := [n] \}
      $$
      a factor-connected flag with respect to $M$.
      Then the matroid $(M|_{S_{j+1}})/S_{j-1}$ has at most two connected components
      for $1 \le j \le k$.
      \begin{enumerate}[(i)]
         \item
         If it has one connected component, the flag
         $$
         \mathcal{G} = \{ \emptyset =: S_0 \subseteq \cdots \subseteq S_{j-1}
         \subseteq S_{j+1} \subseteq \cdots \subseteq S_{k+1} := [n] \}
         $$
         is factor-connected and $Q(M_{\mathcal{G}})$ covers $Q(M_{\flag})$
         in the face lattice of $Q(M)$.
         \item
         If it has two connected components, then they are $S_j \setminus S_{j-1}$ and
         $S_{j+1} \setminus S_j$.
         Moreover, the flag
         $$
         \flag' = \{ \emptyset =: S_0 \subseteq \cdots \subseteq S_{j-1}
         \subseteq S'_j \subseteq S_{j+1} \subseteq \cdots \subseteq S_{k+1} := [n] \},
         $$
         where $S'_j = S_{j-1} \cup (S_{j+1} \setminus S_j)$, is factor-connected
         and $Q(M_{\flag'}) = Q(M_{\flag})$.
         In this case, $\flag$ and $\flag'$ are said to be \emph{adjacent}.
      \end{enumerate}
   \end{lemma}
   
   \begin{proof}
      For $j = 1, \dots, k$, we have
      $(M|_{S_{j+1}})/S_j = [(M|_{S_{j+1}})/S_{j-1}]/(S_j\setminus S_{j-1})$ 
      and $(M|_{S_j})/S_{j-1} = [(M|_{S_{j+1}})/S_{j-1}]|_{S_j}$.
      The first assertion is obtained from \cite[Proposition 4.2.10]{Oxley}, and
      the other assertions follow from \cite[Proposition 4.2.13]{Oxley}
      and Proposition~\ref{bases-of-M-F}.
   \end{proof}
   
   Two factor-connected flags $\flag$ and $\flag'$ of the same length are said to be \emph{equivalent} 
   if there is a sequence of factor-connected flags 
   $$
   \flag = \flag_0, \flag_1, \dots, \flag_r = \flag'
   $$ 
   such that $\flag_i$ is adjacent to $\flag_{i-1}$ for $i = 1, \dots, r$.
   We write $\flag \sim \flag'$ when factor-connected flags $\flag$ and $\flag'$ are equivalent.
   
   
   \begin{lemma}
   \label{switch-connected-case}
      Let $M$ be a connected matroid on $[n]$ and $X, Y$ disjoint subsets of $[n]$
      with $X \cup Y = [n]$.
      Then $\mathcal{B}((M|_X) \oplus (M/X))$ and $\mathcal{B}((M|_Y) \oplus (M/Y))$
      are disjoint.
   \end{lemma}
   
   \begin{proof}
      Since $X \cap Y = \emptyset$ and $X \cup Y = [n]$, 
      one has $r(X) + r(Y) \ge r(M)$.
      If $r(X) + r(Y) = r(M)$, then $|B \cap X| = r(X)$ and $|B \cap Y| = r(Y)$
      for all $B \in \mathcal{B}(M)$, and hence $M = M|_X \oplus M|_Y$.
      But then $M$ is not connected.
      Thus $r(X) + r(Y) > r(M)$.
      Therefore there is no base which has $r(X)$ elements in $X$ and $r(Y)$ elements in $Y$,
      i.e., $\mathcal{B}((M|_X) \oplus (M/X))$ and $\mathcal{B}((M|_Y) \oplus (M/Y))$
      are disjoint.
      %
   \end{proof}
   
   The following proposition shows that the equivalence classes of factor-connected
   flags characterize faces of a matroid base polytope.
   
   \begin{proposition}
   \label{equivalence-classes-of-fcf}
      Let $M$ be a connected matroid on $[n]$.
      If $\flag$ and $\flag'$ are two factor-connected flags of subsets of $[n]$ given by
      $$
      \begin{aligned}
         \flag =& \{ \emptyset =: S_0 \subseteq S_1 \subseteq \cdots 
                  \subseteq S_k \subseteq                  S_{k+1} := [n] \},\\
         \flag' =& \{ \emptyset =: T_0 \subseteq T_1 \subseteq \cdots 
                  \subseteq T_l \subseteq T_{l+1} := [n] \},
      \end{aligned}
      $$
      then $M_{\flag} = M_{\flag'}$ if and only if
      $\flag$ and $\flag'$ are equivalent.
   \end{proposition}
   
   \begin{proof}
      If $\flag \sim \flag'$, then $M_\flag = M_{\flag'}$
      from Lemma~\ref{switching-components}.
      
      For the other direction, suppose that $M_{\flag} = M_{\flag'}$.
      Then $\flag$ and $\flag'$ have the same length since $\dim Q(M_{\flag}) = n-k-1$ 
      and $\dim Q(M_{\flag'}) = n-l-1$.
      Without loss of generality, we may assume that $S_1 \ne T_1$ and $S_k \ne T_k$.
      
      We will use induction on $k$.
      We claim that
      \begin{eqnarray}
      \label{difference-are-the-same}
         T_1 = S_m \setminus S_{m-1} \text{ for some } m > 1.
      \end{eqnarray}
      This claim is going to be used both in the base case and the inductive step.
      Let $m$ be the smallest index such that $T_1 \cap S_m \ne \emptyset$
      and let $\alpha \in T_1 \cap S_m$.
      Suppose that $T_1$ is not contained in $S_m$, i.e.,
      there is an element $\beta \in T_1 \setminus S_m$.
      Since $M|_{T_1}$ is connected, there are
      $B'_1, \widetilde{B}'_1 \in \mathcal{B}(M|_{T_1})$ such that $\alpha \in B'_1$ and
      $\widetilde{B}'_1 = (B'_1 \setminus \{ \alpha \}) \cup \{ \beta \}$.
      Choose $B'_j \in \mathcal{B}((M|_{T_j})/T_{j-1})$ for $j = 2, \dots, k+1$.
      Then 
         $B' = B'_1 \cup B'_2 \cup \cdots \cup B'_{k+1}$ 
         and
         $\widetilde{B}' = \widetilde{B}'_1 \cup B'_2 \cup \cdots \cup B'_{k+1}$ 
      are bases of $M_{\flag'}$.
      But either $B'$ or $\widetilde{B}'$ is not a base of $M_{\flag}$ since
      $$
      |\widetilde{B}' \cap (S_m \setminus S_{m-1})| = |B' \cap (S_m \setminus S_{m-1})|-1.
      $$
      This contradicts the assumption $M_\flag = M_{\flag'}$.
      Therefore $T_1 \subseteq S_m$.

      Now suppose $(S_m \setminus S_{m-1}) \setminus T_1 \ne \emptyset$, i.e.,
      there is $\gamma \in (S_m \setminus S_{m-1}) \setminus T_1$.
      Since $(M|_{S_m})/S_{m-1}$ is connected, there are
      bases $B_m , \widetilde{B}_m$ of $(M|_{S_m})/S_{m-1}$ satisfying
      $\alpha \in B_m$ and $\widetilde{B}_m = (B_m \setminus \{ \alpha \}) \cup \{ \gamma \}$.
      If $B_j \in \mathcal{B}((M|_{S_j})/S_{j-1})$ for all $j \ne m$,
      then 
      $$
      \begin{aligned}
         B &= B_1 \cup \cdots \cup B_{k+1} 
         \textrm{ and}\\
         \widetilde{B} &= B_1 \cup \cdots \cup B_{m-1} \cup \widetilde{B}_m \cup B_{m+1} \cup \cdots \cup B_{k+1}
      \end{aligned}
      $$
      are contained in $\mathcal{B}(M_\flag)$.
      But at least one of them is not in $\mathcal{B} (M_{\flag'})$ because
      $|\widetilde{B} \cap T_1| = |B \cap T_1|-1$.
      This is impossible because $M_\flag = M_{\flag'}$.
      Therefore 
      $T_1 = S_m \setminus S_{m-1}$.
         
      \emph{Base case}: $k = 1$. 
      Equation~(\ref{difference-are-the-same}) shows that $T_1 = S_2 \setminus S_1$. 
      Since $T_1 \cup S_1 = [n]$,  
      $M$ is not connected by Lemma~\ref{switch-connected-case}.
      Thus, $\flag$ and $\flag'$ are equivalent by Lemma~\ref{switching-components}(ii).
      
      \emph{Inductive step}.
      Now suppose $k > 1$.
      By Equation~(\ref{difference-are-the-same}), $T_1 = S_m \setminus S_{m-1}$ for some $m > 1$.
      Let $\widetilde{\flag}$ be a flag
      $$
      \emptyset =: S_0 \subseteq S_1
      \subseteq S_1 \cup T_1 \subseteq S_2 \cup T_1 \subseteq \cdots
      \subseteq \underbrace{S_{m-1} \cup T_1}_{= S_m} \subseteq S_{m+1} \cdots \subseteq S_{k+1} := [n].
      $$
      We claim that $M|_{S_j \cup T_1} = M|_{S_j} \oplus M|_{T_1}$ for $j = 1, \dots, m-1$.
      Let $B_j \in \mathcal{B}(M|_{S_j})$ and $B'_1 \in \mathcal{B}(M|_{T_1})$.
      Since 
      $$
      M_\flag = \oplus_{i=1}^{k+1} (M|_{S_i})/S_{i-1} 
      = \oplus_{i=1}^{k+1} (M|_{T_i})/T_{i-1} = M_{\flag'}
      $$ 
      and $M|_{T_1} = (M|_{S_m})/S_{m-1}$,
      there is a base $B \in \mathcal{B}(M_\flag)$ such that $B \cap S_j = B_j$
      and $B \cap T_1 = B'_1$. 
      Thus $B_j \cup B'_1 \in \mathcal{B}(M|_{S_j \cup T_1})$,
      and hence 
      $$
      r(M|_{S_j \cup T_1}) = r(M|_{S_j}) + r(M|_{T_1}).
      $$
      By \cite[Lemma 3]{MartinReiner}, we have
      $M|_{S_j \cup T_1} = M|_{S_j} \oplus M|_{T_1}$.
      Therefore
      $$
      \begin{aligned}
         (M|_{S_1 \cup T_1})/S_1 &= [M|_{T_1} \oplus M|_{S_1}]/S_1\\
                                 &= (M|_{S_m})/S_{m-1}
      \end{aligned}
      $$
      and
      $$
      \begin{aligned}
         (M|_{S_i \cup T_1})/(S_{i-1} \cup T_1) 
         &= [M|_{T_1} \oplus M|_{S_i}]/(S_{i-1} \cup T_1))\\
         &= (M|_{S_i})/S_{i-1}
      \end{aligned}
      $$
      for $i = 2, \dots, m-1$.
      Thus $\widetilde{\flag}$ is factor-connected and $M_{\widetilde{\flag}} = M_{\flag}$.
      By the induction assumption, $\widetilde{\flag} \sim \flag$.
      Similarly, one can show $S_1 = T_l \setminus T_{l-1}$ for some $l > 1$
      and the chain $\widetilde{\flag}'$ given by
      $$
      \emptyset =: T_0 \subseteq T_1
      \subseteq S_1 \cup T_1 \subseteq S_1 \cup T_2 \subseteq \cdots
      \subseteq \underbrace {S_1 \cup T_{l-1}}_{= T_l} \subseteq T_{l+1} \cdots \subseteq T_{k+1} := [n]
      $$
      is factor-connected and $M_{\widetilde{\flag}'} = M_{\flag'}$.
      Thus the induction assumption gives $\widetilde{\flag}' \sim \flag'$.
      
      By the induction assumption again, we have $\widetilde{\flag} \sim \widetilde{\flag}'$
      and hence $\flag$ and $\flag'$ are equivalent.
   \end{proof}
   
   If $M$ is a matroid on $[n]$ and $S$ is a subset of $[n]$, 
   then the hyperplane $H_S$ defined by $\sum_{e \in S} x_e = r(S)$
   is a supporting hyperplane of $Q(M)$ and $Q(M) \cap H_S$ is
   the matroid base polytope for $(M|_S) \oplus (M/S)$.
   The next lemma tells us when a face of $Q(M)$ is contained in $H_S$.
   
   \begin{lemma}
   \label{lem-set-in-flag}
      Let $M$ be a matroid on $[n]$ and $S$ be a subset of $[n]$.
      A face $\sigma$ of $Q(M)$ is contained in $H_S$ if and only if
      there is a factor-connected flag
      $$
      \flag = \{ \emptyset =: S_0 \subseteq S_1 \subseteq \cdots 
              \subseteq S_k \subseteq S_{k+1} := [n] \}
      $$
      containing $S$ and $\sigma = Q(M_\flag)$.
   \end{lemma}
   
   \begin{proof}
      If there is a flag $\flag$ containing $S$ and $\sigma = Q(M_\flag)$,
      then every base of $M_\flag$ is contained in $\mathcal{B}((M|_S) \oplus (M/S))$
      and hence $\sigma$ is contained in $H_S$.
      
      For the converse, suppose $\sigma$ is contained in $H_S$.
      Since $\sigma$ is a face of $Q(M)$, there is a factor-connected flag 
      $$
      \flag' = \{ \emptyset =: T_0 \subseteq T_1 \subseteq \cdots 
               \subseteq T_k \subseteq T_{k+1} := [n] \},
      $$
      such that $\sigma = Q (M_{\flag'})$ and $k = n-\dim \sigma -1$.
      We claim that
      $T_i \setminus T_{i-1}$ is contained in either $S$ or $[n]\setminus S$ for $i = 1, \dots, k+1$.
      Suppose there are $\alpha \in (T_i \setminus T_{i-1}) \cap S$ and $\beta \in (T_i \setminus T_{i-1}) \setminus S$.
      Since $(M|_{T_i})/T_{i-1}$ is connected, there exist bases $B_1$ and $B_2$ of $M_{\flag'}$
      with $\alpha \in B_1$ and $B_2 = (B_1 \setminus \{ \alpha \}) \cup \{ \beta \}$.
      Then $|B_1 \cap S| > |B_2 \cap S|$ and so $\sigma$ is not contained in $H_S$, which is a contradiction.
      
      Construct a flag $\flag$ as follows:
      $S_j = S_{j-1} \cup (T_l \setminus T_{l-1})$, where $l$ is the smallest index such that 
      \begin{enumerate}[(i)]
         \item
         $T_l \setminus T_{l-1} \subseteq S \setminus S_{j-1}$ if $S \setminus S_{j-1} \ne \emptyset$,
         \item
         $T_l \setminus T_{l-1} \subseteq [n]\setminus S_{j-1}$ if $S \setminus S_{j-1} = \emptyset$.
      \end{enumerate}
      By the construction of $\flag$, $S$ is contained in $\flag$.
      Since $\sigma$ is contained in $H_S$, 
      $$
      M_\flag = \oplus_{i=1}^{k+1} (M|_{S_i})/S_{i-1} = \oplus_{i=1}^{k+1} (M|_{T_i})/T_{i-1} =  M_{\flag'}
      $$ 
      and hence $\sigma = Q(M_{\flag})$.
   \end{proof}
   
   For a face $\sigma$ of $Q(M)$, let $L_\sigma$ be the poset of all subsets of $[n]$ 
   which are contained in some factor-connected flag $\flag$ with $\sigma = Q(M_\flag)$ 
   ordered by inclusion.
   The following lemma shows that $L_\sigma$ is a lattice.
   
   \begin{lemma}
   \label{lem-closed-under-intersection}
      Let $M$ be a matroid on $[n]$ and $S, T \subseteq [n]$.
      If $Q(M)$ is contained in $H_S$ and $H_T$, then it is also contained in $H_{S \cap T}$.
   \end{lemma}
   
   \begin{proof}
      Since $Q(M)$ is contained in $H_S$ and $H_T$, we have
      $$
      r(S \cup T) + r(S \cap T) = r(S) + r(T)
      $$ 
      (see \cite[Lemma 1.3.1]{Oxley}).
      Suppose that $Q(M)$ is not contained in $H_{S \cap T}$, i.e., there is a base $B$
      of $M$ such that $|B \cap (S \cap T)| < r(S \cap T)$.
      Then
      $$
      \begin{aligned}
         |B \cap (S \cup T)| &= |B \cap S| + |B \cap T| - |B \cap (S \cap T)|\\
                             &> r(S) + r(T) - r(S \cap T) = r(S \cup T)
      \end{aligned}
      $$
      which is impossible.
      Therefore $Q(M)$ is contained in $H_{S \cap T}$.
   \end{proof}
   %
   
   Since $L_\sigma$ is a sublattice of the Boolean lattice $\mathcal{B}_n$,
   it is distributive.
   The fundamental theorem for finite distributive lattices~\cite{StanleyI} 
   shows that there is a finite poset $P_\sigma$
   for which $L_\sigma$ is the lattice of order ideals of $P_\sigma$.
   The following theorem describes the poset $P_\sigma$ in terms of matroid data.
   
   \begin{theorem}
   \label{theorem-posets-for-faces-of-matroid-polytopes}
      Let $M$ be a matroid on $[n]$ and $\sigma$ be a face of $Q(M)$.
      Then $L_\sigma$ is the lattice of order ideals of $P_\sigma$, 
      where $P_\sigma$ is the poset defined as follows:
      \begin{enumerate}[(i)]
         \item
         The elements of $P_\sigma$ are the connected components of $M_\sigma$, and
         \item
         for distinct connected components $C_1$ and $C_2$ of $M_\sigma$,
         $C_1 < C_2$ if and only if 
         $$
         C_2 \subseteq S \subseteq [n] \text{ and }
         \sigma \subseteq H_S \text{ implies } C_1 \subseteq S.
         $$
      \end{enumerate}
   \end{theorem}
   
   Note that $P_\sigma$ is a well-defined poset.
   Reflexivity and transitivity are clear.
   Suppose $C_1$ and $C_2$ are distinct connected components of $M_\sigma$ with $C_1 < C_2$.
   Consider a minimal subset $S$ such that $C_2 \subseteq S$ and $\sigma \subseteq H_S$.
   $C_1 < C_2$ implies $C_1$ is contained in $S$. Since $C_2$ is a connected component of $M_\sigma$, 
   Lemma~\ref{lem-set-in-flag} implies that there is a flag containing $S$ and $S \setminus C_2$
   such that $Q(M_\flag) = \sigma$.
   Since $C_1 \subseteq S \setminus C_2$, $\sigma \subseteq H_{S \setminus C_2}$,
   and $C_2 \nsubseteq S \setminus C_2$, we have $C_2 \nless C_1$.
   
   \begin{example}
      Let $M_{2,1,1}$ be the rank $2$ matroid on $[4] = \{ 1, 2, 3, 4 \}$ 
      whose unique non-base is  $12$ and let $\sigma$ be an edge of $Q(M_{2,1,1})$
      connecting $e_{14}$ and $e_{24}$.
      Then the connected components of $M_\sigma$ are $12$, $3$ and $4$.
      Since $\{ 1, 2, 3, 4 \}$ is the only subset $S$ containing $\{ 3 \}$
      such that $\sigma \subseteq H_S$, $12 < 3$ and $4 < 3$ in $P_\sigma$.
      One can see that there are no other relations in $P_\sigma$.
      Figure~\ref{fig-Face-poset} is the proper part of the face poset of $Q(M_{2,1,1})$
      whose faces (shown on the right-hand side of each box) are labeled by corresponding posets 
      (shown on the left-hand side of each box) and $P_\sigma$ is shown in the shaded box.
   \end{example}
   
   \begin{figure}[!t]
      \begin{center}
         \includegraphics[width=0.9\textwidth]{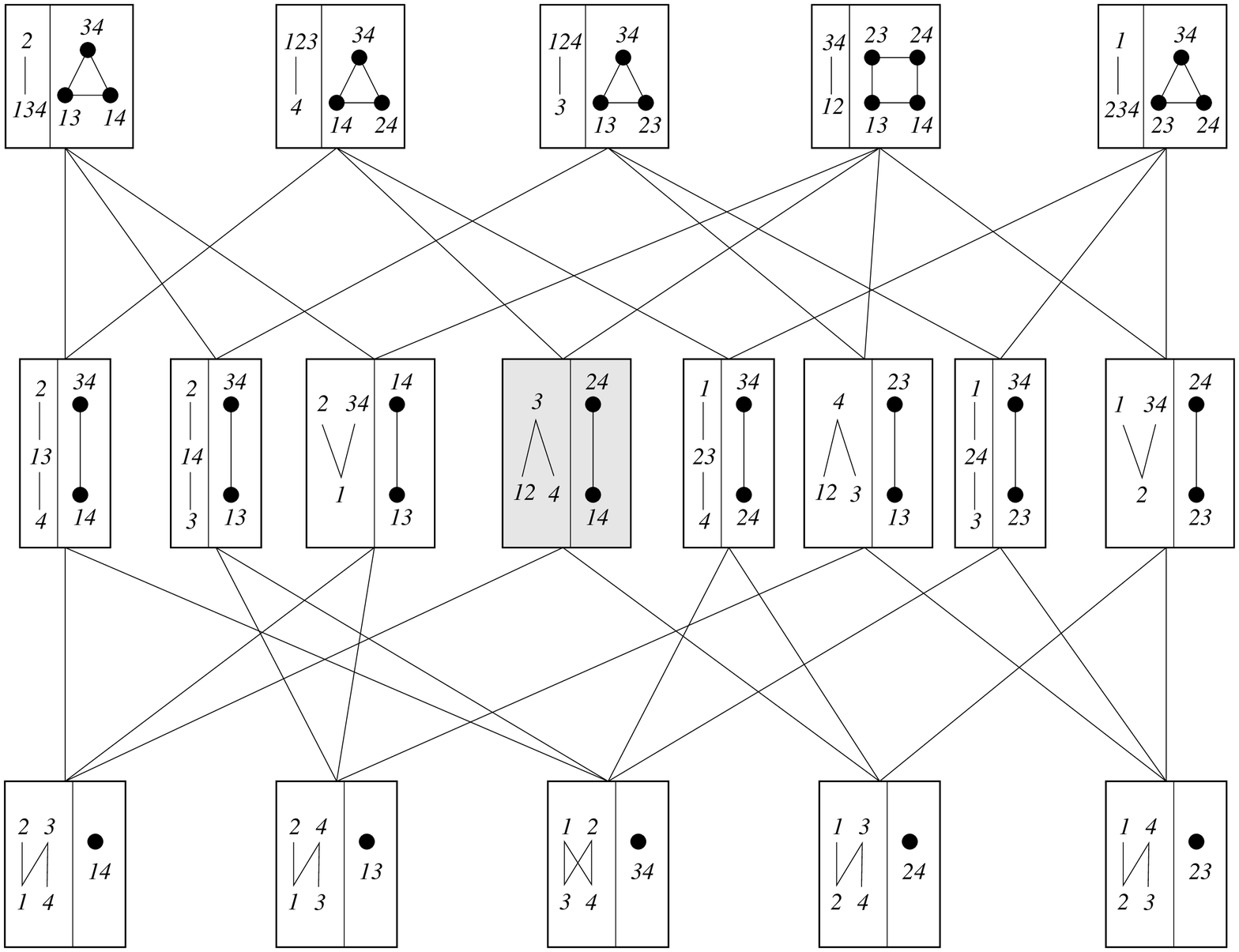}
      
         \begin{caption}
            {The proper part of the face poset of $Q(M_{2,1,1})$}
            \label{fig-Face-poset}
         \end{caption}
      \end{center}
   \end{figure}
   
   \begin{proof}[Proof of Theorem~\ref{theorem-posets-for-faces-of-matroid-polytopes}]
      Let $S$ be a set in $L_\sigma$.
      By Lemma~\ref{lem-set-in-flag}, there is a factor-connected flag
      $$
      \flag = \{ \emptyset =: S_0 \subseteq S_1 \subseteq \cdots 
              \subseteq S_k \subseteq S_{k+1} := [n] \}
      $$
      such that $M_\sigma = M_\flag$ and $S_i = S$ for some $i$.
      Then $S = \cup_{j=1}^i (S_j \setminus S_{j-1})$.
      Suppose $S_l \setminus S_{l-1} < S_j \setminus S_{j-1}$ in $P_\sigma$ for some $j \le i$.
      Since $\sigma \subseteq H_S$, $S_l \setminus S_{l-1} \subseteq S$ and 
      hence $S$ is an order ideal of $P_\sigma$.
      
      Conversely, suppose $T$ is an order ideal of $P_\sigma$.
      Let $U$ be the intersection of all subsets $\widetilde{T}$ such that $\sigma \subseteq H_{\widetilde{T}}$
      and $T \subseteq \widetilde{T}$.
      By Lemma~\ref{lem-closed-under-intersection}, $U$ lies in $L_\sigma$.
      Suppose $T \ne U$.
      By Lemma~\ref{lem-set-in-flag}, there is a factor-connected flag
      $$
      \flag = \{ \emptyset =: U_0 \subseteq U_1 \subseteq \dots
              \subseteq U_k \subseteq U_{k+1} := [n] \}
      $$
      such that $\sigma = Q(M_\flag)$ and $U = U_j$ for some $j$.
      Since $T \ne U$ and $T$ is an order ideal, 
      we may choose $\flag$ so that $U_j \setminus U_{j-1} \nsubseteq T$.
      Then $U_{j-1} \in L_\sigma$ and $T \subseteq U_{j-1}$
      which is a contradiction.
   \end{proof}
   
   The posets $P_\sigma$ coincide with the posets obtained from preposets corresponding to normal cones
   of the matroid base polytope $Q(M)$ studied by Postnikov, Reiner, and Williams~\cite{PostnikovReinerWilliams}.
   
   Billera, Jia, and Reiner~\cite{BilleraJiaReiner} define posets
   related to bases of matroids:
   For each $e \in B \in \mathcal{B}(M)$ the \emph{basic bond for $e$ in $B$}
   is the set of $e' \in [n]\setminus B$ for which $(B \setminus \{ e \}) \cup \{ e' \}$ is another
   base of $M$.
   Dually, for each $e' \in [n]\setminus B$ the \emph{basic circuit for $e$ in $B$} is 
   the set of $e \in B$ for which $(B \setminus \{ e \}) \cup \{ e' \}$ is another base of $M$.
   Since $e'$ lies in the basic bond for $e$ if and only if 
   $e$ lies in the basic circuit of $e'$,
   one can define the poset $P_B$ to be the poset defined on $[n]$ 
   where $e <_{P_B} e'$ if and only if $e'$ is in the basic bond for $e$.
   The next proposition shows that our poset $P_\sigma$ is the same as
   the poset $P_B$ of Billera, Jia, and Reiner if $\sigma$ is a vertex $e_B$.
   
   \begin{proposition}
      If $B$ is a base of a matroid $M$ on $[n]$ and $\sigma$ is a vertex $e_B$ of $Q(M)$, then
      $P_\sigma = P_B$.
   \end{proposition}
   
   \begin{proof}
      Since all connected components of $M_\sigma$ are singletons, 
      $P_\sigma$ and $P_B$ have the same elements.
      If $e_1, e_2 \in B$, then they are not comparable in $P_\sigma$ since
      $e_B \subseteq H_{e_1}$ and $e_B \subseteq H_{e_2}$. 
      Also, $e_1, e_2 \notin B$ are not comparable in $P_\sigma$
      because $e_B \subseteq H_{B \cup \{ e_1 \}}$ and $e_B \subseteq H_{B \cup \{ e_2 \}}$.
      Now assume $e_1 \in B$ and $e_2 \notin B$.
      We claim that $e_1 <_{P_\sigma} e_2$ if and only if $B' := (B \setminus \{ e_1 \}) \cup \{ e_2 \} \in \mathcal{B}(M)$,
      i.e., $e_1 <_{P_B} e_2$.
      Assume $e_1 <_{P_\sigma} e_2$, i.e., $e_2 \in S$ and $\sigma \subseteq H_S$ implies $e_1 \in S$.
      Suppose $B'$ is not a base of $M$.
      Then $r(B') = r(M) -1$.
      Thus we have $e_2 \in B', \sigma \subseteq H_{B'}$, but $e_1 \notin B'$, which is a contradiction.
      Therefore $B'$ is a base of $M$.
      Conversely, assume $B' \in \mathcal{B}(M)$.
      Let $S$ be a subset of $[n]$ such that $e_2 \in S$ and $\sigma \subseteq H_S$, i.e., $|B \cap S| = r(S)$.
      If $e_1 \notin S$, then $|B' \cap S| = |B \cap S| + 1 = r(S) + 1$, which is impossible.
      Thus $e_1 <_{P_\sigma} e_2$.
   \end{proof}
   
\section{The $\C\D$-index}
\label{sec-cd-index}

   In this section, we define the $\C\D$-index for Eulerian posets and give 
   the relationship among $\C\D$-indices of polytopes when a polytope is split by a hyperplane.
   
   Let $P$ be a graded poset of rank $n+1$ with the rank function $\rho$.
   For a subset $S$ of $[n]$, define $f_P (S)$ to be the number of chains of $P$ 
   whose ranks are exactly given by the set $S$.
   The function $f_P : 2^{[n]} \to \NN$ is called the \emph{flag f-vector} of $P$.
   The \emph{flag h-vector} of $P$ is defined by the identity
   $$
   h_P (S) = \sum_{T \subseteq S} (-1)^{|S \setminus T|} \cdot f_P(T).
   $$
   Since  this identity is equivalent to the relation
   $$
   f_P (S) = \sum_{T \subseteq S} h_P (T),
   $$
   the flag $f$-vector and the flag $h$-vector contain the same information.
   
   For $S \subseteq [n]$, define the \emph{noncommutative} $\A\B$-monomial $u_S = u_1 \dots u_n$,
   where
   $$
   u_i = \left\{
   \begin{array}{ll}
      \A & \text{if } i \notin S, \\
      \B & \text{if } i \in S.
   \end{array}
   \right.
   $$
   The $\A\B$-\emph{index} of the poset $P$  is defined to be the sum
   $$
   \Psi (P) = \sum_{S \subseteq [n]} h_P (S) \cdot u_S.
   $$
   
   An alternative way of defining the $\A\B$-index is as follows.
   For a chain 
   $$
   c := \{ \hat{0} < x_1 < \cdots < x_k < \hat{1} \},
   $$
   we give a \emph{weight} $w_P (c) = w(c) = z_1 \cdots z_n$, where
   $$
   z_i = \left\{
   \begin{array}{ll}
      \B & \text{if } i \in \{ \rho(x_1), \dots, \rho(x_k) \},\\
      \A - \B & \text{otherwise.}
   \end{array}
   \right.
   $$
   Define the $\A\B$-\emph{index} of the poset $P$ to be the sum
   $$
   \Psi (P) = \sum_c w(c),
   $$
   where the sum is over all chains $c = \{ \hat{0} < x_1 < \cdots < x_k < \hat{1} \}$
   in $P$.
   Recall that a poset $P$ is \emph{Eulerian} if its M\"{o}bius function $\mu$
   is given by $\mu(x, y) = (-1)^{\rho(y) - \rho(x)}$
   (see \cite{StanleyI} for more details).
   One important class of Eulerian posets is face lattices of convex polytopes 
   (see \cite{Lindstrom,Rota}).
   Fine conjectured and Bayer and Klapper~\cite{BayerKlapper} proved that
   the $\A\B$-index of an Eulerian poset $P$ can be written uniquely
   as a polynomial of $\C = \A + \B$ and $\D = \A\B + \B\A$.
   When the $\A\B$-index can be written as a polynomial in $\C$ and $\D$, we call
   $\Psi (P)$ the $\C\D$-\emph{index} of $P$.
   We will use the notation $\Psi(Q)$ for the $\C\D$-index of the face poset of 
   a convex polytope $Q$.
   
   Let $v$ be a vertex of a polytope $Q$ and let $l(x) = c$ be a supporting hyperplane
   of $Q$ defining $v$.
   The \emph{vertex figure} $Q/v$ of $v$ is defined by
   $$
   Q/v = Q \cap \{ l(x) = c+ \delta \}
   $$
   where $\delta$ is an arbitrary small positive number.
   The polytope $Q/v$ depends on the choice of $l, c$, and $\delta$, but 
   it is well-known that its combinatorial type is independent of these variables.
   For a face $\sigma$ of $Q$, the \emph{face figure} $Q/\sigma$ of $\sigma$ is defined by
   $$
   Q/\sigma = ( \dots (( Q / \sigma_0 ) / \sigma_1 ) \dots ) / \sigma_k ,
   $$
   where $\sigma_0 \subseteq \sigma_1 \subseteq \dots \subseteq \sigma_k = \sigma$ 
   is a maximal chain with $\dim \sigma_i = i$.
   For faces $\sigma$ and $\tau$ of $Q$ with $\sigma \subseteq \tau$, 
   the face lattice of the face figure $\tau / \sigma$ is the interval $[\sigma, \tau]$.
   
   Ehrenborg and Readdy~\cite[Prop. 4.2]{EhrenborgReaddy} give formulas for the $\C\D$-index
   of a pyramid, a prism and a bipyramid of a polytope.
   
   \begin{proposition}
   \label{prop-pyramid-and-prism}
      Let $Q$ be a polytope. Then
      $$
      \begin{aligned}
         \Psi (\pyr (Q))   &= \frac 1 2 \left[ \Psi(Q) \cdot \C + \C \cdot \Psi (Q)
                              + \sum_{\sigma}\Psi (\sigma) \cdot \D \cdot \Psi (Q/\sigma)
                              \right],\\
         \Psi (\prism (Q)) &= \Psi (Q) \cdot \C + 
                              \sum_{\sigma} \Psi (\sigma) \cdot \D \cdot \Psi (Q/\sigma),\\
         \Psi (\bipyr (Q)) &= \C \cdot \Psi (Q) + 
                              \sum_{\sigma} \Psi (\sigma) \cdot \D \cdot \Psi (Q/\sigma),
      \end{aligned}
      $$
      where the sums are over all proper faces $\sigma$ of $Q$.
   \end{proposition}
   
   Note that the $\C\D$-index of $\bipyr (Q)$ is obtained from the $\C\D$-index of $\prism (Q)$ 
   because $\bipyr (Q)$ is the dual of the prism of the dual of $Q$
   and the $\C\D$-index of the dual polytope is obtained by writing every $\A\B$-monomial in reverse order
   (see~\cite{EhrenborgReaddy} for details).
   
   Let $Q$ be a polytope in $\reals^n$.
   Let $H$ be a hyperplane in $\reals^n$ defined by $l(x) = c$ and 
   $H^+$ (resp. $H^-$) be the closed halfspace $l(x) \ge c$ (resp. $l(x) \le c$).
   For simplicity, let $Q^+ := Q \cap H^+$, $Q^- = Q \cap H^-$, and $\widehat{Q} := Q \cap H$.
   Also we use the notations $\sigma^+ := \sigma \cap H^+$, $\sigma^- := \sigma \cap H^-$, and  
   $\hat{\sigma} := \sigma \cap H$ for a face $\sigma$ of $Q$.
   The following theorem provides the relationship among $\C\D$-indices of polytopes 
   $Q$, $Q^+$, $Q^-$ and faces of $\widehat{Q}$ when a polytope $Q$ is split by a hyperplane.
   
   \begin{theorem}
   \label{cd-index-hyperplane-split-case}
      Let $Q$ be a polytope in $\reals^n$ and $H$ be a hyperplane in $\reals^n$ intersecting 
      the interior of $Q$.
      Then the following identity holds:
      $$
      \Psi (Q) = \Psi (Q^+) + \Psi (Q^-) - \Psi (\widehat{Q}) \cdot \C - 
      \sum_\sigma \Psi (\hat{\sigma}) \cdot \D \cdot \Psi ( \widehat{Q} / \hat{\sigma} ),
      $$
      where the sum is over all proper faces $\sigma$ of $Q$ intersecting both
      open halfspaces $H^+ \setminus H$ and $H^- \setminus H$ nontrivially.
   \end{theorem}
   
   \begin{proof}
      Let
      $$
      c = \{ \emptyset = \sigma_0 < \sigma_1 < \cdots < \sigma_k < \sigma_{k+1} = Q \}
      $$
      be a chain of faces of $Q$.
      If there is a face $\sigma_i$ which meets $H^+ \setminus H$
      (resp. $H^- \setminus H$) but does not meet $H^- \setminus H$
      (resp. $H^+ \setminus H$),
      then
      $$
      c^+ = \{ \emptyset = \sigma_0^+ < \sigma_1^+ < \cdots < \sigma_k^+ < \sigma_{k+1}^+ = Q^+ \}
      $$
      $$
      (\text{resp. }c^- = \{ \emptyset = \sigma_0^- < \sigma_1^- < \cdots < \sigma_k^- < \sigma_{k+1}^- = Q^- \})
      $$
      is a corresponding chain in the face poset of $Q^+$ (resp. $Q^-$) and
      $w(c) = w(c^+)$ (resp. $w(c) = w(c^-)$).
      
      Now suppose that either $\sigma_i$ is contained in $H$ or intersects both 
      $H^+ \setminus H$ and $H^- \setminus H$ for all $i = 0, \dots , k+1$.
      Let $s$ be the smallest index such that $\sigma_i \nsubseteq H$ and
      $\hat{c}$ be the chain in the face poset of $\widehat{Q}$ defined by
      $$
      \hat{c} := c \cap H = \{ \emptyset = \sigma_0 < \hat{\sigma}_1 < \cdots
      < \hat{\sigma}_{k+1} = \widehat{Q} \}.
      $$
      
      \begin{enumerate}
         \item
         Consider a chain $c^+$ in the face poset of $Q^+$ such that
         $c^+ \cap H$ and $\hat{c}$ are the same.
         Then $c^+$ is one of the following chains for some $j$ with $s \le j \le k+1$:
         $$
         \begin{aligned}
            c^+_1 &= \{ \emptyset < \hat{\sigma}_1 < \cdots < \hat{\sigma}_{j-1}
                     < \sigma_j^+ < \cdots < \sigma_k^+ < \sigma_{k+1}^+ = Q^+ \},\\
            c^+_2 &= \{ \emptyset < \hat{\sigma}_1 < \cdots < \hat{\sigma}_{j-1}
                     < \hat{\sigma}_j < \sigma_j^+ < \cdots < \sigma_k^+ < \sigma_{k+1}^+ = Q^+ \}.
         \end{aligned}
         $$
         Let $\sigma := \sigma_j$. 
         Then $\sigma$ meets both open halfspaces $H^+ \setminus H$ and $H^- \setminus H$ nontrivially.
         Let 
         $\hat{c}_1 := \{ \emptyset < \hat{\sigma}_1 < \cdots < \hat{\sigma}_{j-1} < \hat{\sigma} \}$ and
         $\hat{c}_2 := \{ \hat{\sigma} < \hat{\sigma}_{j+1} < \cdots < \hat{\sigma}_k < \widehat{Q} \}$.
         There are two cases:
         \begin{enumerate}
            \item
            The first case is when $s \le j \le k$.
            Then the sum of the weights of the chains $c^+_1$ and $c^+_2$ is given by
            $$
            w(c^+_1) + w(c^+_2) = w_{[\emptyset, \hat{\sigma} ]}(\hat{c}_1)\cdot \A \cdot \B \cdot
            w_{[\hat{\sigma}, \widehat{Q}]}(\hat{c}_2).
            $$
            Note that this case does not occur when $s = k+1$.
            \item
            The second case is when $j = k+1$.
            Then the sum of the weights of the chains is
            $$
            w(c^+_1) + w(c^+_2) = w(\hat{c}) \cdot \A.
            $$
         \end{enumerate}
         \item
         Consider a chain $c^-$ in the face poset of $Q^-$ such that
         $c^- \cap H$ and $\hat{c}$ are the same.
         Then $c^-$ is one of the following chains for some $j$ with $s - 1 \le j \le k+1$:
         $$
         \begin{aligned}
            c^-_1 &= \{ \emptyset < \hat{\sigma}_1 < \cdots < \hat{\sigma}_j
                     < \sigma_{j+1}^- < \cdots < \sigma_k^- < \sigma_{k+1}^- = Q^- \},\\
            c^-_2 &= \{ \emptyset < \hat{\sigma}_1 < \cdots < \hat{\sigma}_j
                     < \sigma_j^- < \sigma_{j+1}^- < \cdots < \sigma_k^- < \sigma_{k+1}^- = Q^- \}.
         \end{aligned}
         $$
         Let $\sigma := \sigma_j$ and define chains $\hat{c}_1$ and $\hat{c}_2$ as in Case (1). 
         If $j \ge s$, then $\sigma$ meets both $H^+ \setminus H$ and $H^- \setminus H$.
         There are three cases:
         \begin{enumerate}
            \item
            The first case is when $s \le j \le k$.
            Then the sum of the weights of the chains $c^-_1$ and $c^-_2$ is given by
            $$
            w(c^-_1) + w(c^-_2) = w_{[\emptyset, \hat{\sigma} ]}(\hat{c}_1)\cdot \B \cdot \A \cdot
            w_{[\hat{\sigma}, \widehat{Q}]}(\hat{c}_2).
            $$
            Note that this case does not occur when $s = k+1$.
            \item
            The second case is when $j = k+1$.
            Then the only possible chain is $c^-_2$ and its weight is
            $$
            w(c^-_2) = w(\hat{c}) \cdot \B.
            $$
            \item
            The third case is when $j = s - 1$.
            Since $\sigma_{s-1} \subseteq H$, we have $\hat{\sigma}_{s-1} = \hat{\sigma}^-_{s-1}$.
            Therefore the only possible chain is $c^-_1$ and its weight is
            $$
            w(c^-_1) = w_{[\emptyset, \hat{\sigma} ]}(\hat{c}_1) \cdot \B \cdot (\A - \B) \cdot 
            w_{[\hat{\sigma}, \widehat{Q}]}(\hat{c}_2) = w(c).
            $$
         \end{enumerate}
      \end{enumerate}
      Now summing over all such chains in the face posets of $Q^+$ and $Q^-$, we obtain
      $$
      \sum_{c^+} w(c^+) + \sum_{c^-} w(c^-) =
         \sum_{\hat{c}_1, \hat{c}_2} w_{[\emptyset, \hat{\sigma} ]}(\hat{c}_1)\cdot \D \cdot 
         w_{[\hat{\sigma}, \widehat{Q}]}(\hat{c}_2) 
         + w(\hat{c}) \cdot \C + w(c),
      $$
      where the first (resp. second) sum is over all chains $c^+$ (resp. $c^-$) in the face poset of
      $Q^+$ (resp. $Q^-$) such that $c^+ \cap H = \hat{c}$ (resp. $c^- \cap H = \hat{c}$),
      and the third sum is over all pairs of chains $\hat{c}_1, \hat{c}_2$ obtained above
      such that $\hat{c}_2$ is nontrivial.
      Finally, summing over all chains in the face poset of $Q$, 
      we finish the proof.
   \end{proof}

   \begin{remark}
      The formula for the prism of a polytope in 
      Proposition~\ref{prop-pyramid-and-prism} is a special case of
      Theorem~\ref{cd-index-hyperplane-split-case},
      since in this case 
      $$
      \begin{aligned}
         Q = \prism(Q') &\simeq Q' \times \left[0, 0.5 \right] = Q^- \\
                        &\simeq Q' \times \left[0.5, 1 \right] = Q^+
      \end{aligned}
      $$
      and $Q' \simeq Q' \times \{ 0.5 \} = \widehat{Q}$.
      Also, the formula for the pyramid of a polytope is obtained from 
      Theorem~\ref{cd-index-hyperplane-split-case}
      by considering $Q = \bipyr (Q')$ split by the hyperplane containing $Q'$:
      in this case, $Q^+ = Q^- = \pyr (Q')$ and there are no proper faces of $\bipyr (Q')$
      intersecting both open halfspaces nontrivially.
   \end{remark}
   
\section{Hyperplane splits of a matroid base polytope}
\label{sec-hyperplane-splits}

   In this section, we define hyperplane splits of a matroid base polytope
   and give conditions when they occur.
   
   For a matroid $M$ on $[n]$, a \emph{hyperplane split} of the matroid base polytope $Q(M)$ is a decomposition
   $Q(M) = Q(M_1) \cup Q(M_2)$ where
   \begin{enumerate}[(i)]
      \item
      $M_1$ and $M_2$ are matroids on $[n]$, and
      \item
      the intersection $Q(M_1) \cap Q(M_2)$ is a proper face of 
      both $Q(M_1)$ and $Q(M_2)$.
   \end{enumerate}
   
   Let $\sum_{i=1}^n a_i x_i = b$ be an equation defining corresponding hyperplane $H$.
   Since the intersection $Q(M_1) \cap Q(M_2)$ is a matroid base polytope contained in $H$ and 
   its edges are parallel to $e_i - e_j$ for some $i \ne j$,
   the only constraints on the normal vector $(a_1, a_2, \dots, a_n)$ of $H$ 
   are of the form $a_i = a_j$.
   Using the fact that $Q(M)$ is a subset of a simplex defined by $\sum_{i=1}^n x_i = r(M)$
   and scaling the right hand side $b$, one can assume that $H$ is defined by
   $\sum_{e \in S} x_e = k$ for some subset $S$ of $[n]$ and some positive integer $k < r(M)$.
   
   The following result characterizes when hyperplane splits occur.
   
   \begin{theorem}
   \label{matroid-base-polytope-split}
      Let $M$ be a rank $r$ matroid on $[n]$ and $H$ a hyperplane defined by
      $\sum_{e \in S} x_e = k$.
      Then $H$ gives a hyperplane split of $Q(M)$
      if and only if the following are satisfied:
      \begin{enumerate}[(i)]
         \item
         $r(S) > k$ and $r(S^c) > r-k$,
         \item
         if $I_1$ and $I_2$ are $k$-element independent subsets of $S$ such that
         $(M/I_1)|_{S^c}$ and $(M/I_2)|_{S^c}$ have rank $r-k$, then 
         $(M/I_1)|_{S^c} = (M/I_2)|_{S^c}.$
      \end{enumerate}
   \end{theorem}
   
   \begin{remark}
      Note that if $I$ is a $k$-element independent subset of $S$ and 
      $J$ is an $(r-k)$-element independent subset of $S^c$, then
      $I$ is a base for $(M/J)|_S$ if and only if $J$ is a base for $(M/I)|_{S^c}$.
      Therefore the condition (ii) can be replaced with the following condition for $S^c$:
      \begin{enumerate} \it
         \item[(ii$\,'$)]
         if $J_1$ and $J_2$ are $(r-k)$-element independent subsets of $S^c$ such that
         $(M/J_1)|_S$ and $(M/J_2)|_S$ have rank $k$, then 
         $(M/J_1)|_S = (M/J_2)|_S$.
      \end{enumerate}
   \end{remark}
   
   \begin{remark}
      Herrmann and Joswig \cite{HerrmannJoswig} study the splits of general polytopes.
      For the entire hypersimplex, Theorem~\ref{matroid-base-polytope-split} reduces to \cite[Lemma 38]{HerrmannJoswig}.
   \end{remark}
   
   \begin{proof}[Proof of Theorem~\ref{matroid-base-polytope-split}]
      The condition (i) is equivalent to the condition that $H$ intersects the interior of $Q(M)$ nontrivially.
      
      Define $\mathcal{B}_k = \{ B \in \mathcal{B}(M) : | B \cap S | = k \}$.
      We will show that the condition (ii) holds if and only if
      $\mathcal{B}_k$ is a collection of bases of some matroid.
      Then the assertion follows from Theorem~\ref{thm-matroid-polytopes}.
      
      Suppose that the condition (ii) is true.
      Choose any bases $B_1$ and $B_2$ in $\mathcal{B}_k$ and $x \in B_1 \setminus B_2$ 
      (without loss of generality, we may assume $x \in B_1 \cap S$).
      Let $I_i = B_i \cap S$ and $J_i = B_i \setminus S$ for $i = 1,2$.
      Then the condition (ii) implies that there is a base $B = I_2 \cup J_1$ in $\mathcal{B}_k$.
      Since $B_1, B \in \mathcal{B}$, there is $y \in B \setminus B_1 \subseteq I_2 \subseteq B_2$ 
      such that $B_3 = (B \setminus \{ x \}) \cup \{ y \} \in \mathcal{B}$.
      Since $y \in I_2 \subseteq S$, $B_3 \in \mathcal{B}_k$.
      Thus $\mathcal{B}_k$ forms a collection of bases of a matroid.
      
      Conversely suppose that $\mathcal{B}_k$ is a collection of bases of some matroid.
      Let $I_1$ and $I_2$ be $k$-element independent subsets of $S$ such that
      $(M/I_1)|_{S^c}$ and $(M/I_2)|_{S^c}$ have rank $r-k$.
      Choose $J_1 \in \mathcal{B}((M/I_1)|_{S^c})$ and $J_2 \in \mathcal{B}((M/I_2)|_{S^c})$.
      Then $B_1 = I_1 \cup J_1$ and $B_2 = I_2 \cup J_2$ are bases for $B$.
      We claim that $I_2 \cup J_1$ is also a base of $M$: this implies
      $\mathcal{B}((M/I_1)|_{S^c}) \subseteq \mathcal{B}((M/I_2)|_{S^c})$ and 
      (ii) follows by symmetry.
      We use induction on the size of $I_1 - I_2$.
      
      \emph{Base Case}:
         If $|I_1 \setminus I_2| = 0$, we have $I_2 \cup J_1 = B_1 \in \mathcal{B}$.
      
      \emph{Inductive Step}:
         Suppose $|I_1 \setminus I_2| = l$ for some $l \le k$.
         Choose an element $x \in I_1 \setminus I_2 \subseteq B_1 \setminus B_2$.
         Since $\mathcal{B}_k$ forms a matroid, there exist $y \in I_2 \setminus I_1$ such that
         $B_3 = (B_1 \setminus \{ x \}) \cup \{ y \} \in \mathcal{B}_k \subseteq \mathcal{B}$.
         Since $B_3 = [(I_1 \setminus \{ x \}) \cup \{ y \}] \cup J_1$, we have $|(B_3 \cap S) \setminus I_2| = l-1$
         and the induction hypothesis implies $I_2 \cup J_1 \in \mathcal{B}$.
   \end{proof}
   
\section{Rank $2$ matroids}
\label{sec-rank-2-matroids}

   In this section we study the $\C\D$-index of a matroid base polytope 
   when a matroid has rank $2$.
   
   A (loopless) rank $2$ matroid $M$ on $[n]$ is determined up to isomorphism by the composition
   $\alpha(M)$ of $n$ that gives the sizes $\alpha_i$ of its parallelism classes.
   Let $\alpha := (\alpha_1, \alpha_2, \dots ,\alpha_k)$ be a composition of $n$ with the length $\ell(\alpha) = k$ 
   and let $M_\alpha$ be the corresponding rank $2$ matroid on $[n]$.
   When $k = 2$, we have
   $$
   Q(M_\alpha) \cong Q(U_{1,\alpha_1}) \times Q(U_{1, \alpha_2})
   \cong \DD_{\alpha_1} \times \DD_{\alpha_2},
   $$
   where $U_{r,n}$ is the uniform matroid of rank $r$ on $n$ elements
   and $\DD_n$ is an $(n-1)$-dimensional simplex.
   
   For two weak compositions (i.e., compositions allowing $0$ as parts)
   $\alpha$ and $\beta$ of the same length, we define
   $\beta \le \alpha$ if $\beta_i \le \alpha_i$ for all $i = 1, 2, \dots, \ell(\alpha)$.
   Let $\bar{\beta}$ be the composition obtained from $\beta$ by deleting $0$ parts.
   If $\alpha = (2,4,0,6,7)$ and $\beta = (1,3,0,6,3)$, then $\beta < \alpha$ and
   $\bar{\beta} = (1,3,6,3)$.
   
   When $M$ has rank $2$, Theorem~\ref{matroid-base-polytope-split}
   can be rephrased in the following way. 
   
   \begin{corollary}
   \label{matroid-base-polytope-split-rank-two}
      Let $M$ be a rank $2$ matroid on $[n]$ and $H$ be a hyperplane defined by 
      $\sum_{e \in S} x_e = 1$.
      Then $H$ gives a hyperplane split of $Q(M)$ if and only if
      $S$ and $S^c$ are both unions of at least two parallelism classes of $M$.
   \end{corollary}
   %
      %
   %
   %
   %
   
   \begin{question}
   \label{matroidal-version-of-cd-index-hyperplane-split-case}
      When $Q(M_1) \cup Q(M_2)$ is a hyperplane split of $Q(M)$ with a corresponding
      hyperplane $H$, can Theorem~\ref{cd-index-hyperplane-split-case} be restated 
      in terms of matroids?
   \end{question}
   
   %
   %
   %
   %
   
   If $M$ has rank $2$, then one can rephrase Theorem~\ref{cd-index-hyperplane-split-case} 
   in terms of matroids (see Proposition~\ref{cd-index-hyperplane-split-case-rank-two} below), but 
   Question~\ref{matroidal-version-of-cd-index-hyperplane-split-case} is open for higher ranks.
   
   After the relabeling, one may assume that a rank $2$ matroid $M$ has parallelism classes $P_1, P_2, \dots, P_k$
   and $S = \cup_{i = 1}^m P_i$ for some $m$.
   In this case, one can restate Theorem~\ref{cd-index-hyperplane-split-case} in terms of matroids as follows.
   
   \begin{proposition}
   \label{cd-index-hyperplane-split-case-rank-two}
      Let $M$ be a rank $2$ matroid on $[n]$ with at least four parallelism classes $P_1, P_2, \dots, P_k$
      and $S = \cup_{i = 1}^m P_i$ for some $m$ satisfying $2 \le m \le k-2$.
      Then the hyperplane $H$ defined by $\sum_{e \in S} x_e = 1$ gives a hyperplane split 
      $Q(M_1) \cup Q(M_2)$ of $Q(M)$ where $M_1$ is a matroid with parallelism classes $S, P_{m+1}, \dots, P_k$
      and $M_2$ is a matroid whose parallelism classes are $P_1, \dots, P_m, S^c$.
      Moreover,
      $$
      \begin{aligned}
         \Psi (Q(M)) =& \Psi (Q(M_1)) + \Psi (Q(M_2)) - \Psi (\DD_{|S|} \times \DD_{n-|S|}) \cdot \C \\
                      &- \sum_T \Psi (\DD_{|T \cap S|} \times \DD_{|T \setminus S|}) \cdot \D \cdot \Psi (\DD_{n-|T|}),
      \end{aligned}
      $$
      where the sum is over all proper subsets $T$ of $[n]$ such that
      $M|_T$ has at least four parallelism classes and
      $S$ and $S^c$ both contain at least two parallelism classes of $M|_T$.
   \end{proposition}
   
   \begin{proof}
      Note that there is no flacet (i.e., facet corresponding to a flat of $M$)
      of $Q(M)$ which intersects both open halfspace given by $H$ nontrivially
      since every flacet of $Q(M)$ corresponds to a base set of the form
      $$
      \{ B \in \mathcal{B}(M) : |B \cap P_i| = 1  \}
      $$
      for some $i$.
      (The corresponding flacet has empty intersection with the open halfspace $H^- \setminus H$ if $i \le m$
      while it has empty intersection with $H^+ \setminus H$ otherwise.)
      If $\sigma$ is a face of $Q(M)$ which has nonempty intersection with both open halfspaces given by $H$,
      then $\sigma$ is the intersection of some facets of $Q(M)$ which are not flacets.
      Since each facet of $Q(M)$ which is not a flacet corresponds to the deletion of an element of $[n]$,
      $\sigma$ corresponds to a matroid $M|_T$ for some subset $T$ of $[n]$.
      Also, the upper interval $[\sigma \cap H, Q(M) \cap H]$ $(\cong [\sigma, Q(M)])$ is isomorphic to 
      the Boolean algebra of order $n - |T|$, 
      and hence $\Psi[(Q(M) \cap H) / (\sigma \cap H)] = \Psi (\DD_{n-|T|})$.
      Finally, $\sigma$ has nonempty intersection with both open halfspaces given by $H$ if and only if
      $M|_T$ has at least four parallelism classes and 
      $S$ and $S^c$ both contain at least two parallelism classes of $M|_T$.
      Now, the result follows from Theorem~\ref{cd-index-hyperplane-split-case}.
   \end{proof}
   
   The following corollary, which is obtained from Corollary~\ref{matroid-base-polytope-split-rank-two}
   and Proposition~\ref{cd-index-hyperplane-split-case-rank-two},
   expresses the $\C\D$-index of a matroid base polytope
   of a rank $2$ matroid $M$ 
   in terms of $\C\D$-indices of matroid base polytopes of matroids corresponding to compositions of length $\le 3$.
   For simplicity, we use the following notations:
   $$
   \begin{aligned}
      \lambda(\alpha, i) &= \left( \sum_{j=1}^{i-1} \alpha_j, \alpha_i, \sum_{j=i+1}^{\ell(\alpha)} \alpha_j \right)
                         &\text{ for } 2 \le i \le \ell(\alpha) - 1, \\
      \mu(\alpha, i)     &= \left(\sum_{j=1}^{i} \alpha_j, \sum_{j=i+1}^{\ell(\alpha)} \alpha_j \right)
                         &\text{ for } 1 \le i \le \ell(\alpha) - 1.
   \end{aligned}
   $$
   If $\alpha = (2,4,0,6,7)$, then
   $\lambda(\alpha, 4) = (6,6,7)$ and $\mu(\alpha, 4) = (12,7)$.
   For a rank $2$ matroid $M$ with parallelism classes $P_1, \dots, P_k$, 
   define $M_i$ be the rank $2$ matroid with two parallelism classes $P_1 \cup \dots \cup P_i$ and $P_{i+1} \cup \dots \cup P_k$.
   Note that $(M_\alpha)_i$ is isomorphic to $M_{\mu(\alpha, i)}$.
      
   \begin{corollary}
   \label{cd-index-of-rank-2-matroids}
      Let $\alpha$ be a composition of $n$ with at least three parts and 
      $M_\alpha$ be the corresponding rank $2$ matroid on $[n]$.
      Then the $\C\D$-index of $Q(M_\alpha)$ can be expressed as follows:
      $$
      \begin{aligned}
         \Psi(Q(M_\alpha)) =& \sum_{i = 2}^{\ell(\alpha)-1} \Psi(Q(M_{\lambda(\alpha, i)})) 
                              - \left( \sum_{i=2}^{\ell(\alpha)-2} 
                              \Psi(\DD_{\mu(\alpha,i)_1} \times \DD_{\mu(\alpha,i)_2}) \right) \cdot \C \\
                            &- \sum_{\substack{\beta < \alpha\\ \ell(\bar{\beta}) \ge 4}} \prod_{j=1}^{\ell(\alpha)} {\binom{\alpha_j}{\beta_j}} 
                             \left( \sum_{i=2}^{\ell(\bar{\beta})-2} 
                             \Psi(\DD_{\mu(\bar{\beta}, i)_1} \times \DD_{\mu(\bar{\beta},i)_2}) \right)
                             \cdot \D \cdot \Psi(\DD_{n-|\bar{\beta}|}).
      \end{aligned}
      $$
   \end{corollary}
   
   %
   \begin{proof}
      %
%
      For simplicity, let $M := M_\alpha$ and $k := \ell (\alpha)$.
      After the relabeling, one may assume that the matroid $M$ has parallelism classes $P_1, \dots, P_k$
      with $|P_i| = \alpha_i$ for all $i = 1, \dots, k$.
      We first claim that 
      $$
      \begin{aligned}
      \Psi(Q(M)) = & \sum_{i = 2}^{k-1} \Psi(Q(M_{\lambda(\alpha, i)})) 
                     - \left( \sum_{i=2}^{k-2} \Psi(\DD_{\mu(\alpha,i)_1} \times \DD_{\mu(\alpha,i)_2}) \right) \cdot \C \\
                   &- \sum_S \left( \sum_{i=2}^{p(S)-2} 
                             \Psi(Q((M|_S)_i) \right) \cdot \D \cdot \Psi(\DD_{n-|S|}).
      \end{aligned}
      $$
      where the sum in the second line runs over all proper subsets $S$ of $[n]$
      such that $M|_S$ has at least four parallelism classes,
      and $p(S)$ is the number of parallelism classes of $M|_S$.
      We will use induction on $k$.
      
      \emph{Base Case}:
         If $k=3$, then the result is clear.
      
      \emph{Inductive Step}:
         By Corollary~\ref{matroid-base-polytope-split-rank-two},
         the hyperplane $H$ defined by 
         $$
         \sum_{e \in P_1 \cup \dots \cup P_{k-2}} x_e = 1
         $$ 
         gives a hyperplane split $Q(M^+) \cup Q(M^-)$ of $Q(M)$ where 
         $M^+$ is the matroid whose parallelism classes are $P_1, \dots, P_{k-2}, P_{k-1} \cup P_k$
         and the matroid $M^-$ has parallelism classes $P_1 \cup \dots \cup P_{k-2}, P_{k-1}, P_k$.
         Since $M^-$ is isomorphic to $M_{\lambda(\alpha, k-1)}$,
         Proposition~\ref{cd-index-hyperplane-split-case-rank-two} implies
         $$
         \begin{aligned}
            \Psi (Q(M)) =& \Psi (Q(M^+)) + \Psi (Q(M_{\lambda(\alpha, k-1)})) 
                           - \Psi (\DD_{\mu(\alpha, k-2)_1} \times \DD_{\mu(\alpha, k-2)_2}) \cdot \C \\
                         &- \sum_T \Psi (Q((M|_T)_{p(T)-2})) \cdot \D \cdot \Psi (\DD_{n-|T|}),
         \end{aligned}
         $$
         where the sum is over all proper subsets $T$ of $[n]$ such that
         $M|_T$ has at least four parallelism classes, two of which are subsets of 
         $P_{k-1}$ and $P_k$ respectively.
         Since 
         $M^+$ has $k-1$ parallelism classes, the induction hypothesis implies the claim.

      Since $M|_S$ corresponds to the composition 
      $$\beta_T := (|P_1 \cap S|, |P_2 \cap S|, \dots, |P_k \cap S|) < \alpha$$
      and there are 
      $$\prod_{j=1}^k \binom{\alpha_j}{\beta_j}$$
      subsets corresponding to $\beta < \alpha$,
      we finish the proof.
   \end{proof}
   
   Purtill~\cite{Purtill} shows that the $\C\D$-index of the simplex $\DD^n$ is 
   the $(n+1)$-st Andr\'{e} polynomial.
   Using the formula for the $\C\D$-index of a product of two polytopes given
   by Ehrenborg and Readdy~\cite{EhrenborgReaddy}, one can calculate
   terms in the second and the third sums in Corollary~\ref{cd-index-of-rank-2-matroids}.
   We still don't have a simple interpretation for the $\C\D$-index for $Q(M_\alpha)$
   when $\alpha$ has three parts.
   The first few values of the $\C\D$-index for $Q(M_\alpha)$, where $\alpha$ has three parts,
   are displayed in Table~\ref{table-cd-index-for-indecomposable-matroids}.
   
   \begin{table}[!ht]
   \begin{center}
      \begin{tabular}{|c|l|}
         \hline
         $\alpha$    & $\Psi (Q(M_\alpha))$ \\
         \hline \hline
         $(1,1,1)$ & $\C^2 + \D$ \\ 
         \hline
         $(2,1,1)$ & $\C^3 + 3 \C \D + 3 \D \C$ \\
         \hline
         $(3,1,1)$ & $\C^4 + 4 \C^2 \D + 8 \C\D\C + 5 \D \C^2 + 7 \D^2$ \\
         $(2,2,1)$ & $\C^4 + 5 \C^2 \D + 10 \C\D\C +  6 \D \C^2 + 10 \D^2$ \\
         \hline
         $(4,1,1)$ & $\C^5 + 5 \C^3 \D + 13 \C^2 \D \C + 15 \C \D \C^2 + 20 \C \D^2 + 7 \D \C^3 + 18 \D \C \D + 22 \D^2 \C$ \\
         $(3,2,1)$ & $\C^5 + 6 \C^3 \D + 17 \C^2 \D \C + 20 \C \D \C^2 + 28 \C \D^2 + 9 \D \C^3 + 26 \D \C \D + 33 \D^2 \C$ \\
         $(2,2,2)$ & $\C^5 + 7 \C^3 \D + 21 \C^2 \D \C + 24 \C \D \C^2 + 36 \C \D^2 + 10 \D \C^3 + 34 \D \C \D + 42 \D^2 \C$ \\
         \hline
         $(5,1,1)$ & $\C^6 + 6 \C^4 \D + 19 \C^3 \D \C + 29 \C^2 \D \C^2 + 38 \C^2 \D^2 + 24 \C \D \C^3 + 60 \C \D \C \D$ \\
                   & \qquad $+ 72 \C \D^2 \C + 9 \D \C^4 + 33 \D \C^2 \D + 65 \D \C \D \C + 47 \D^2 \C^2 + 64 \D^3$ \\
         
         $(4,2,1)$ & $\C^6 + 7 \C^4 \D + 24 \C^3 \D \C + 39 \C^2 \D \C^2 + 52 \C^2 \D^2 + 33 \C \D \C^3 + 86 \C \D \C \D$ \\
                   & \qquad $+ 104 \C \D^2 \C + 12 \D \C^4 + 48 \D \C^2 \D + 98 \D \C \D \C + 72 \D^2 \C^2 + 100 \D^3$ \\
         $(3,3,1)$ & $\C^6 + 7 \C^4 \D + 25 \C^3 \D \C + 42 \C^2 \D \C^2 + 55 \C^2 \D^2 + 36 \C \D \C^3 + 93 \C \D \C \D$ \\
                   & \qquad $+ 114 \C \D^2 \C + 13 \D \C^4 + 52 \D \C^2 \D + 109 \D \C \D \C + 81 \D^2 \C^2 + 112 \D^3$ \\
         $(3,2,2)$ & $\C^6 + 8 \C^4 \D + 30 \C^3 \D \C + 51 \C^2 \D \C^2 + 69 \C^2 \D^2 + 42 \C \D \C^3 + 116 \C \D \C \D$ \\
                   & \qquad $+ 142 \C \D^2 \C + 14 \D \C^4 + 64 \D \C^2 \D + 136 \D \C \D \C + 98 \D^2 \C^2 + 142 \D^3$ \\
         \hline
      \end{tabular}
      \begin{caption}
         {$\C\D$-index for $Q(M_\alpha)$ for composition $\alpha$ with three parts}
         \label{table-cd-index-for-indecomposable-matroids}
      \end{caption}
   \end{center}
   \end{table}

\def\cprime{$'$}





\end{document}